\documentclass[11pt,reqno]{amsart}

\usepackage{amssymb, amsmath, amsthm}
\usepackage{hyperref}
\usepackage{verbatim}
\usepackage{amscd}   
\usepackage[all]{xy} 
\usepackage{youngtab} 
\usepackage{young} 
\usepackage{ytableau}
\usepackage{tikz-cd}
\usepackage{cases}
\usepackage{array}
\usepackage{tabu}
\usepackage{calligra,mathrsfs}

\title{Minimal free resolutions of ideals of minors associated to pairs of matrices}

\author{Andr\'as Cristian L\H{o}rincz}
\address{Max Planck Institute for Mathematics in the Sciences, Inselstrasse 22, Leipzig, Germany 04103}
\email{lorincz@mis.mpg.de}


 \newtheorem{theorem}{Theorem}[section]
  
 \newtheorem{lemma}[theorem]{Lemma}
 
 \newtheorem{prop}[theorem]{Proposition}
 
 \newtheorem{example}{Example}[section]

\theoremstyle{remark}

\DeclareMathOperator{\ShHom}{\mathscr{H}\text{\kern -3pt {\calligra\large om}}\,}
\DeclareMathOperator{\Dim}{\mathbf{dim}}

\newcommand{\A}{\mathbb{A}}
\newcommand{\D}{\mathbb{D}}

\newcommand{\nat}{\mathbb{N}}

\renewcommand{\a}{\alpha}
\renewcommand{\b}{\beta}
\newcommand{\bo}{\bigoplus}

\newcommand{\E}{\mathcal{E}}
\newcommand{\bw}{\bigwedge}

\newcommand{\kk}{\Bbbk}
\renewcommand{\ll}{\lambda}

\newcommand{\oo}{\otimes}

\renewcommand{\SS}{\mathbb{S}}

\renewcommand{\L}{\mathcal{L}}

\newcommand{\R}{\mathcal{R}}
\newcommand{\B}{\mathcal{B}}
\newcommand{\Q}{\mathcal{Q}}
\newcommand{\M}{\mathcal{M}}

\newcommand{\Ext}{\operatorname{Ext}}
\newcommand{\Fl}{\operatorname{Flag}}
\newcommand{\GL}{\operatorname{GL}}

\newcommand{\Hom}{\operatorname{Hom}}
\newcommand{\End}{\operatorname{End}}

\newcommand{\Rep}{\operatorname{Rep}}

\newcommand{\rank}{\operatorname{rank}}

\newcommand{\Gr}{\operatorname{Gr}}

\renewcommand{\ker}{\operatorname{ker}}
\newcommand{\im}{\operatorname{Im}}

\newcommand{\ol}[1]{\overline{#1}}

\begin{document}

 \begin{abstract} 
Consider the affine space consisting of pairs of matrices $(A,B)$ of fixed size, and its closed subvariety given by the rank conditions $\rank A \leq a, \, \rank B \leq b$ and $\rank (A\cdot B) \leq c$, for three non-negative integers $a,b,c$. These varieties are precisely the orbit closures of representations for the equioriented $\A_3$ quiver. In this paper we construct the (equivariant) minimal free resolutions of the defining ideals of such varieties. We show how this problem is equivalent to determining the cohomology groups of the tensor product of two Schur functors of tautological bundles on a 2-step flag variety. We provide several techniques for the determination of these groups, which is of independent interest. 
\end{abstract}

\maketitle

\section*{Introduction}\label{sec:intro}

In the seminal paper \cite{lasc}, Lascoux determined the minimal free resolutions for the determinantal varieties. This has led generalizations in various directions, such as the Kempf--Lascoux--Weyman geometric technique for calculating minimal free resolutions of other varieties linked to representation theory (see \cite{jerzy}).

In this paper we construct the (equivariant) minimal free resolutions of the defining ideals of orbit closures of the equioriented $\A_3$ quiver based on the Kempf--Lascoux--Weyman technique. This approach has been used in \cite{kavita2,kavita,bubfree} to determine minimal free resolutions of 1-step orbit closures of quivers. For the $\A_2$ and non-equioriented $\A_3$ quivers all representations are 1-step (see \cite{kavita2}). However, this fails for the equioriented $\A_3$ quiver \cite[Section 1]{bubfree}. Because of this, dealing with the case of the equioriented $\A_3$ quiver in this context is substantially more difficult than the non-equioriented case.

In our case, we show that the problem of determining the terms in the minimal free resolutions is equivalent to computing the cohomology of some vector bundles on a 2-step flag variety. With the optimal choice of desingularizations, such bundles can be written as a tensor products of two Schur functors -- one applied to a tautological subbundle and the other to a tautological quotient bundle. The difficulty stems from the fact that the Borel--Weil--Bott Theorem is not directly applicable to such bundles -- this is a consequence to the failure of the 1-step property. Therefore, we devote a considerable part of the paper to provide methods of computation for the cohomology of such bundles. These bundles are not semisimple, and as such, it is important to study them as a first step towards the general problem of determining the cohomology of equivariant vector bundles on flag varieties that are not semisimple. Such problems have been studied for Grassmannians in \cite{ottaviani}.

As a consequence of our calculations, we recover that orbit closures of the equioriented $\A_3$ quivers have rational singularities (hence, are normal and Cohen--Macaulay), and we describe explicitly the minimal generators of their defining ideals.
Geometric properties of orbit closures of quivers have been studied extensively, and it is an active area of research (see
\cite{expo} for an exposition).  It has been shown (see \cite{abea,orb1,orb2,ryan,lak}) that for quivers of type $\A$ and $\D$ orbit closures have rational singularities. Furthermore, for equioriented type $\A$ quivers it was shown in \cite{lak} that singularities of orbit closures are identical to singularities of Schubert varieties. Other results regarding singularities of varieties of quiver representations can be found in \cite{bubsing, bfun, bubdecomp} for zero sets of semi-invariants, and in \cite{radical} for quivers with nodes.


The article is organized as follows. In Section \ref{sec:quiv} we recall some facts about the representation theory of the equioriented $\A_3$ quiver, then construct the desingularizations of its orbit closures that are the most suitable for our calculations. In Section \ref{sec:bott} we introduce some basic notation for partitions and recall the Borel--Weil--Bott theorem. In Section \ref{sec:geom}, we apply the Kempf--Lascoux--Weyman geometric technique for the chosen desingularizations, reformulating the problem of finding minimal free resolutions in terms of the bundles mentioned above (see Proposition \ref{prop:conclude}). In Section \ref{sec:coho} we discuss three methods to compute the cohomology of such bundles, the ones in Sections \ref{sec:refineschur} and \ref{sec:definschur} based on Schur complexes. In Section \ref{sec:applications} we apply these methods to compute the minimal free resolutions of the defining ideals of orbit closures, and we describe explicitly the minimal generators of these ideals.

\section{Preliminaries}\label{sec:prelim}

 Throughout we work over a field $\kk$ of characteristic $0$.

\subsection{Quivers}\label{sec:quiv}
A quiver $Q$ is an oriented graph, i.e. a pair $Q=(Q_0,Q_1)$ formed by a finite set of vertices $Q_0$ and a finite set of arrows $Q_1$. An arrow $\a$ has a head $h\a$, and tail $t\a$, that are elements in $Q_0$:

\[\xymatrix{
t\a \ar[r]^{\a} & h\a
}\]

A representation $V$ of $Q$ is a family of finite dimensional vector spaces $\{V_x\,|\, x\in Q_0\}$ together with linear maps $\{V(\a) : V_{t\a}\to V_{h\a}\, | \, \a\in Q_1\}$. The dimension vector $\Dim V\in \nat^{Q_0}$ of a representation $V$ is the tuple $\Dim V=(\dim V_x)_{x\in Q_0}$. A morphism $\phi:V\to W$ of two representations $V,W$ is a collection of linear maps $\phi = \{\phi(x) : V_x \to W_x\,| \,x\in Q_0\}$, with the property that for each $\a\in Q_1$ we have $\phi(h\a)V(\a)=W(\a)\phi(t\a)$. 

We form the affine space of representations with dimension vector $\mathbf{d}\in \nat^{Q_0}$ by
$$\Rep(Q,\mathbf{d}):=\displaystyle\bigoplus_{\a\in Q_1} \Hom(\kk^{d_{t\a}},\,\kk^{d_{h\a}}).$$
The group 
$$\GL(\mathbf{d}):= \prod_{x\in Q_0} \GL(d_x)$$
acts by conjugation on $\Rep(Q,\mathbf{d})$ in the obvious way. Under the action $\GL(\mathbf{d})$ two elements lie in the same orbit if and only if they are isomorphic as representations.

From now on $Q$ denotes the equioriented $\A_3$ quiver

\[\xymatrix{
1 \ar[r]^\a & 2 \ar[r]^\b &3
}\]

It is known (see \cite{elements}) that $Q$ has (up to isomorphism) six indecomposable representations: the simples $S_1,S_2,S_3$, the injective cover $I_2$, the projective cover $P_2$ and the injective-projective $I_3$. The dimension vectors of these indecomposables are $(1,0,0),\, (0,1,0),\, (0,0,1),\, (1,1,0),\, (0,1,1),\, (1,1,1)$, respectively. 

Fix a dimension vector $\mathbf{d} = (d_1,d_2,d_3) \in \nat^3$. Fix  a representation $V\in \Rep(Q,\mathbf{d})$, which by the above has a decomposition

\begin{equation}\label{eq:dec}
V \, \cong \, S_1^{a_1}  \oplus S_2^{a_2} \oplus S_3^{a_3} \oplus I_2^{b_1} \oplus P_2^{b_2} \oplus I_3^c,
\end{equation}
for some $a_1,a_2,a_3, b_1,b_2,c \in \nat$. For convenience, we describe the Auslander-Reiten quiver of $Q$ (see \cite{elements}) together with a diagram recording the multiplicities introduced:

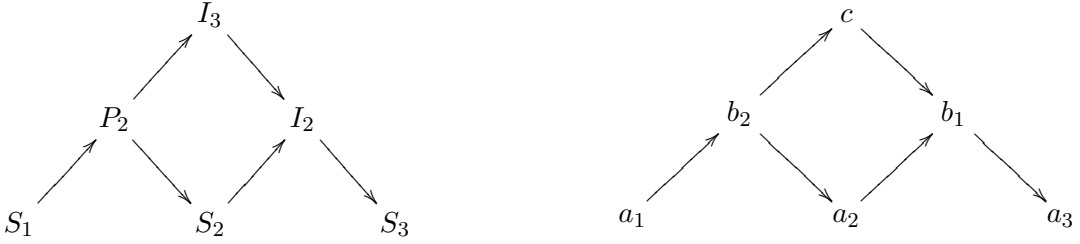
\[ \xymatrix@C-0.5pc{
& & I_3 \ar[dr] & &\\
& P_2 \ar[ur]\ar[dr] & & I_2 \ar[dr] &\\
S_1 \ar[ur] & & S_2 \ar[ur] & & S_3
} \hspace{1in}
\xymatrix{
& & c \ar[dr] & &\\
& b_2 \ar[ur]\ar[dr] & & b_1 \ar[dr] &\\
a_1 \ar[ur] & & a_2 \ar[ur] & & a_3
}
\]

We have the following equations involving the multiplicities:
\begin{equation}\label{eq:equa}
\begin{split}
d_1 = a_1 + b_1 + c,\,\,\,\, d_2 = a_2 + b_1 + b_2 + c, \,\,\,\, d_3 = a_3 + b_2 + c,\\
 \rank \, V(\a) = b_1 + c, \,\,\,\, \rank \, V(\b) = b_2 + c, \,\,\,\, \rank \, V(\b) \circ V(\a) = c.
\end{split}
\end{equation}
In particular, the isomorphism class of a representation is completely determined by the ranks of $V(\a), V(\b)$ and $V(\b)\circ V(\a)$, and orbit closures are indeed the same as the rank varieties mentioned in the Introduction (see \cite{abeasis}).

In case of Dynkin quivers, Reineke \cite{reineke} constructs desingularizations for all orbit closures. These are total spaces of some vector bundles over a product of flag varieties. Inspired by this, we construct desingularizations that make calculations via the KLM geometric technique as accessible as possible (more precisely, see Proposition \ref{prop:conclude}).

For non-negative integers $r_1 \leq r_2 \leq n$ we denote by $\Fl(r_1, r_2, n)$ the 2-step flag variety consisting of flags of spaces $R_1 \subset R_2 \subset \kk^n$ with $\dim R_i=r_i$. Similarly, for $r\leq m$ we denote by $\Gr(r,m)$ the Grassmannian consisting of subspaces $R\subset \kk^m$ with $\dim R = r$.

Take $V$ as in (\ref{eq:dec}), and consider its orbit closure $\ol{O}_V$. Consider the variety
\[\Rep(Q,\,\mathbf{d}) \times \Fl(b_1+c, \, d_2 - b_2, \, d_2) \times \Gr(c,\, d_3), \]
viewed as a trivial bundle over $\Fl(b_1+c, \, d_2 - b_2, \, d_2) \times \Gr(c, d_3)$, and let $Z$ denote the subset consisting of elements of the form $(V,\, R_1 \subset R_2,\, R)$ such that 
\[ \im V(\a) \subseteq R_1 \,\,\,\, \mbox{and} \,\,\,\, V(\b)(R_2) \subseteq R.\]
Clearly, $Z$ is a subbundle, and projection to the first factor gives a proper map $q: Z \to \ol{O}_V$ (see (\ref{eq:equa})).

\begin{prop}\label{prop:desing}
The map $q: Z \to \ol{O}_V$ constructed above is a resolution of singularities.
\end{prop}

\begin{proof}
We need to show only that $q^{-1}(V)$ is a point. Let $(R_1 \subset R_2,\, R)$ be a point in the fiber. Since $\rank V(\a) = b_1 + c$, we must have $R_1 = \im V(\a)$. Since $\rank V(\b)\circ V(\a) = c$, we must have $R=\im V(\b)\circ V(\a)$. Lastly, since $\rank V(\b) = b_2+c$, we obtain $R_2 = R_1 + \ker V(\b)$.
\end{proof}

\subsection{Borel--Weil--Bott Theorem}\label{sec:bott}

A partition (with $r$ parts) $\lambda = (\lambda_1,\,\lambda_2,\dots,\,\lambda_r)$ is a non-increasing sequence of non-negative integers. Sometimes we omit writing the zero entries of partitions. For a partition $\lambda$ we associate its corresponding Young diagram that consists of $\lambda_i$ boxes in the $i$th row. We denote the number of boxes by $|\lambda|:= \lambda_1+\dots+\lambda_r$. We denote by $u_\lambda$ the size of the Durfee square of $\lambda$, that is, the biggest square fitting inside of the Young diagram of $\lambda$. Its defining property is $\lambda_{u_\lambda}\geq u_\lambda$ and $\lambda_{u_\lambda+1}\leq u_\lambda$, which also makes sense for any non-increasing sequence of integers $\ll$.

For a partition $\ll$, we denote by $\lambda^+$ be the partition $(\lambda_1-u_{\lambda},\,\lambda_2-u_{\lambda},\dots,\,\lambda_{u_\lambda}-u_{\lambda})$ and by $\lambda^{-}$ the partition $(\lambda_{u_\lambda+1},\dots,\,\lambda_r)$. Pictorially, we can view the Young diagram of $\lambda$ as the composite of the partitions $\lambda^+, \,\lambda^{-}$ and a $u_\lambda\times u_\lambda$ square as follows:
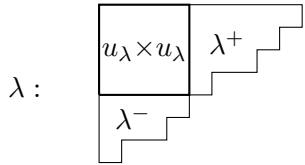
\[ \hspace{-0.3in} \lambda: \hspace{0.2in}
\begin{aligned}
\begin{tikzpicture}
\draw (0,0.9)--(0,0)--(0.3,0)--(0.3,0.3)--(0.9,0.3)--(0.9,0.6)--(1.2,0.6)--(1.2,0.9);
\node at (0.45,0.6) {$\lambda^-$};
\draw [thick] (0,0.9)--(0,2.1)--(1.2,2.1)--(1.2,0.9)--(0,0.9);
\node at (0.6,1.5) {$u_\lambda \hspace{-0.2pc} \times \hspace{-0.2pc} u_\lambda$};
\draw(1.2,0.9)--(1.5,0.9)--(1.5,1.2)--(2.1,1.2)--(2.1,1.5)--(2.4,1.5)--(2.4,1.8)--(2.7,1.8)--(2.7,2.1)--(1.2,2.1);
\node at (1.7,1.6) {$\lambda^+$};
\end{tikzpicture}
\end{aligned}
\]
The conjugate partition $\lambda'=(\lambda_1',\dots, \,\lambda_m')$ is a partition with $\lambda'_i$ being the number of boxes in the $i$th column of the Young diagram of $\lambda$. Given another partition $\mu$, we write $\mu \subset \ll$ if the Young diagram of $\ll$ contains the diagram of $\mu$. Denote by $-\lambda$ the non-increasing sequence of non-positive integers $(-\lambda_r,\,-\lambda_{r-1},\dots,\, -\lambda_1)$.

A weight (with $n$ parts) is any sequence of integers  $\delta:=(\delta_1,\, \delta_2,\dots ,\,\delta_n)$. Consider the action of the symmetric group $\Sigma_n$ on weights defined as follows: a transposition $\sigma_i=(i,i+1)$ acts according to the exchange rule
$$\sigma_i \cdot \delta = (\delta_1,\dots ,\,\delta_{i-1},\,\delta_{i+1}-1,\,\delta_{i}+1,\,\delta_{i+2},\dots,\,\delta_n).$$ 
Let $N(\delta)$ be length of the (unique) permutation $\sigma\in \Sigma_n$ such that the sequence $\sigma \cdot \delta$ is non-increasing, if there exists such a permutation, otherwise put $N(\delta):=-\infty$. Equivalently, $N(\delta)$ is the minimal number of exchanges applied to $\delta$ that turn it non-increasing. 

For any partition $\ll$, we denote by $\SS_\ll$ the corresponding Schur functor (see \cite{jerzy}). We recall the tautological sequence of bundles on the Grassmannian $\Gr(r,n)$ (here $W=\kk^n$)
\[ 0 \to \R \to W \times \Gr(r,n) \to \Q \to 0.\]

In order to compute the cohomology of the bundle $\SS_{\ll} \mathcal{R} \otimes \SS_{\mu} \mathcal{Q}^*$ on $\Gr(r,n)$, we apply the Borel--Weil--Bott theorem (see \cite[Corollary 4.1.7, Corollary 4.1.9]{jerzy}). Namely, consider the weight
$$\delta:=(-\mu,\ll),$$
where $-\mu$ has $n-r$ parts and $\ll$ has $r$ parts (appending with zeroes, if necessary).
\begin{theorem}\label{thm:bott}
The cohomology $H^i(\Gr(r,n), \SS_{\ll} \mathcal{R} \otimes \SS_{\mu} \mathcal{Q}^*)$ vanishes when $i\neq N(\delta)$, and 
$$H^{N(\delta)}(\Gr(r,n), \SS_{\ll} \mathcal{R} \otimes \SS_{\mu} \mathcal{Q}^*)=\SS_{\tau(\delta)} W,$$
where $\tau(\delta)$ is the non-increasing sequence obtained from $\delta$.
\end{theorem}

\section{The KLW geometric technique}
\label{sec:geom}

\vspace{0.05in}

In this section, we apply the Kempf--Lascoux--Weyman geometric technique for the equioriented $\A_3$ quiver. For more on the geometric technique, see \cite{jerzy}. 

Let $V\in \Rep(Q,\mathbf{d})$ be as in (\ref{eq:dec}), and consider the desingularization $q: Z \to \ol{O}_V$ from Proposition \ref{prop:desing}. We make the identification
\[\Rep(Q,\mathbf{d})\,=\, V^*_{1}\otimes V_{2} \,\,\oplus\,\, V_{2}^* \otimes V_3.\]
Let $\R_1 \subset \R_2$ denote the tautological subbundles on $\Fl(b_1+c, d_2 - b_2, c)$. Let $\R$ denote the tautological subbundle on $\Gr(c, d_3)$ and $\Q=V_3/\R$ the tautological quotient bundle. For simplicity, write $X=\Fl(b_1+c, d_2 - b_2, d_2) \times \Gr(c, d_3)$.

We view $\Rep(Q,\alpha) \times X$ as the total space of the trivial bundle $\E$ and $Z$ as the total space of some subbundle $\mathcal{S}$ of $\E$. Let $\xi$ denote the dual of the factorbundle $\E/\mathcal{S}$. More explicitly, it is given by the following locally free sheaf on $X$:
\begin{equation}\label{eq:ksi}
\xi \, = \, V_{1}\otimes (V_2/\R_1)^*\, \oplus\,\, \R_2 \otimes \Q^*.
\end{equation}

As in \cite[Theorem 5.1.2]{jerzy}, we consider a complex $F_\bullet$ with terms (here $A=\kk[\Rep(Q,\mathbf{d})]$)
\begin{equation}\label{eq:basic}
F_i = \bigoplus_{j\geq 0} \, H^j (X, \bigwedge^{i+j} \xi)\otimes A(-i-j).
\end{equation}

As we can see in Theorem \ref{thm:rational}, this gives the minimal free resolution of the defining ideal of $\ol{O}_V$.  In order to apply  \cite[Theorem 5.1.3]{jerzy}, we need to evaluate the cohomologies of the exterior powers $\bw^t \xi$, for $t\geq 0$.  For $1\leq t \leq \dim \xi$, we decompose $\bigwedge^{t} \xi$ using Cauchy's formula (see \cite{jerzy}):
\[\bigwedge^{t} \xi = \bo_{t_1+t_2=t}\,  \bw^{t_1} (V_1 \oo (V_2/\R_1)^*) \oo \bw^{t_2} (\R_2 \oo \Q^*)= \bo_{\substack{\ll, \, \mu \\ |\ll|+|\mu| = t}}\!\!\! \SS_{\mu'} V_1 \oo \SS_{\mu} (V_2/\R_1)^* \oo \SS_{\ll} \R_2 \oo \SS_{\ll'} \Q^*.\]
Hence, in order to describe the complex (\ref{eq:basic}), we need to compute for given partitions $\ll,\mu$:
\[ H^j \bigl(X,\,\, \SS_{\mu'} V_1 \oo \SS_{\mu} (V_2/\R_1)^* \oo \SS_{\ll} \R_2 \oo \SS_{\ll'} \Q^*\bigr) \cong \]
\begin{equation}\label{eq:mult}
\cong \SS_{\mu'} V_1 \oo H^{j-c \cdot u_\ll} \bigl( \Fl(b_1+c, \, d_2 - b_2, \, d_2),\, \SS_{\mu} (V_2/\R_1)^* \oo \SS_{\ll} \R_2 \bigr) \oo H^{c \cdot u_\ll}(\Gr(c,d_3), \,\SS_{\ll'} \Q^*),
\end{equation}
where $u_\ll$ is the Durfee size of $\ll$. We used Theorem \ref{thm:bott} to see that $\SS_{\ll'}\Q^*$ can have cohomology only in degree $c \cdot u_\ll$. In fact, putting $u:=u_\ll$ the cohomology is non-zero if and only if  $\ll_{u+c} \geq u$, when we have
\begin{equation}\label{eq:durf}
H^{c\cdot u}(\Gr(c,d_3),\,\SS_{\ll'} \Q^*)=\SS_{\big{(\ll_1'-c, \, \ll_2'-c, \dots,\, \ll'_u-c, \, u, \dots, \,u, \, \ll_{u+1}',\dots, \, \ll_{d_3-c}')}} V_3^*.
\end{equation}

For simplicity, put $V_2 = W, \, r_1=b_1+c, \, r_2=d_2-b_2, \, n=d_2$. We summarize the above discussion.

\begin{prop}\label{prop:conclude}
The problem of determining the (equivariant) terms of the minimal free resolutions for all orbit closures of the equioriented $\A_3$ quiver is equivalent to the problem of determining the cohomology (as representations of $\GL(n)$) of the bundles on $\Fl(r_1,r_2,n)$ of the form
\[ \SS_\ll \R_2 \oo \SS_\mu (W/\R_1)^*,\]
for all $0\leq r_1 \leq r_2 \leq n= \dim W$ and partitions $\ll,\mu$.
\end{prop}

\begin{proof}
As seen above, determining the cohomology groups of $\SS_\ll \R_2 \oo \SS_\mu (W/\R_1)^*$ (as representations of $\GL(n)$) gives the equivariant terms of the minimal free resolution $F_\bullet$ of orbit closures (see Theorem \ref{thm:rational}). 

Conversely, assume we know the equivariant terms of $F_\bullet$ for the case when $c=0$. Pick two partitions $\ll,\mu$.  Look at all representations in the term $F_i$ of the form 
\[\SS_{\mu'} V_1 \oo \SS_{\gamma} V_2 \oo \SS_{\ll'} V_3,\]
with $\gamma$ a non-increasing sequence of integers. Collect all such $\gamma$ in a set $\Gamma$ (counted with multiplicities).
From (\ref{eq:mult}) we get that 
\[H^{|\ll|+|\mu|-i}(\SS_\ll \R_2 \oo \SS_\mu (W/\R_1)^*) = \, \bigoplus_{\gamma  \in \Gamma}\, \SS_\gamma W.\]
\end{proof}

\section{Cohomology of the tensor product of Schur functors of tautological bundles}\label{sec:coho}

Fix $r_1\leq r_2 \leq n$ and consider the flag variety $X=\Fl(r_1,r_2,n)$. Let $\R_1,\R_2$ be the tautological subbundles of the trivial bundle $W$, with $\dim \R_i = r_i$ and $\dim W=n$. The goal in this section is to provide methods to compute the cohomology of bundles encountered in the previous section, namely
\[ \SS_\ll \R_2 \oo \SS_\mu (W/\R_1)^*,\]
for partitions $\ll,\mu$. We can assume that $\ll$ has at most $r_2$ parts, and $\mu$ has at most $n-r_1$ parts. The symmetry between the partitions $\ll$ and $\mu$ is as follows.

\begin{lemma}\label{lem:symm}
For any $i\geq 0$, we have a $\GL(W)$-isomorphism
\[H^i\bigl(\,\Fl(r_1,\, r_2, \,n), \, \SS_\ll \R_2 \oo \SS_\mu (W/\R_1)^*\, \bigr) \cong H^i\bigl( \,\Fl(n-r_2,\, n- r_1, \, n), \,\SS_\mu \R_2 \oo \SS_\ll (W/\R_1)^*\, \bigr)^*.\]
\end{lemma}

\begin{proof}
This follows by working on the dual space $W^*$ instead of $W$, where $(W/\R_1)^*$ becomes a tautological subbundle.
\end{proof}


\subsection{Splitting method}\label{sec:split}

In this section we consider what is perhaps the simplest approach. Namely, we take the exact sequences
\begin{equation}\label{eq:sequence}
 0 \to \R_1 \to \R_2 \to \R_2/\R_1\to 0   \,\,\, \mbox{ and } \,\,\, 0 \to (W/\R_2)^*\to (W/\R_1)^* \to (\R_2/\R_1)^* \to 0.
\end{equation}
Consider the respective split bundles 
\[\B_1=\SS_\ll (\R_1 \oplus \R_2/\R_1) \oo \SS_\mu (W/\R_1)^* \,\,\, \mbox{ and } \,\,\,  \B_2=\SS_\ll \R_2 \oo \SS_\mu (W/\R_2\oplus \R_2/\R_1)^*.\]
We can compute the cohomologies of $\B_1$ (resp. $\B_2$)  using the Littlewood--Richardson rule (e.g. \cite[Section 2.3]{jerzy}) and the (relative) Borel--Weil--Bott theorem. Let us describe the latter for the case of $\B_1$.

Let $\pi : \Fl(r_1, r_2, n) \to \Gr(r_1,n)$ be the map obtained by forgetting the space of dimension $r_2$. Then for partitions $\gamma,\nu$ we have (by abuse of notation, let $\pi^*( \SS_\mu(W/\R_1)^*) =  \SS_\mu(W/\R_1)^*$ and $\pi^*(\SS_{\nu}\R_1)= \SS_{\nu}\R_1$\,)
\begin{equation}\label{eq:proj}
\mathbb{R}^\bullet \pi_* \big(\SS_\gamma(\R_2/\R_1) \oo \SS_{\nu}\R_1\oo \SS_\mu(W/\R_1)^*\big) =\mathbb{R}^\bullet \pi_* \left(\SS_\gamma (\R_2/\R_1)\right) \oo  \SS_{\nu}\R_1\oo \SS_\mu(W/\R_1)^*,
\end{equation}
by the projection formula. Moreover, by the (relative) Borel-Weil-Bott theorem (see \cite[Theorem 4.1.8]{jerzy}) we have $\mathbb{R}^i \pi_* \SS_\gamma(\R_2/\R_1)=0$ for all $i\neq u_\gamma \cdot (n-r_2)$, and (when $\gamma_{u_\gamma}\geq u_\gamma + n-r_2$)
\begin{equation}\label{eq:push}
\mathbb{R}^{u_\gamma \cdot (n-r_2)} \pi_* ( \, \SS_\gamma(\R_2/\R_1) ) = \SS_{(\gamma_1 - (n-r_2), \dots, \,  \gamma_{u_\gamma}- (n-r_2), \, u_\gamma^{\,n-r_2}, \, \gamma^-)}W/\R_1.
\end{equation}
Since the derived pushforward lives in a single degree, we now can calculate cohomology on $\Gr(r_1,n)$ and use Theorem \ref{thm:bott} to obtain the cohomologies of $\B_1$. 

The bundle $\SS_\ll \R_2 \oo \SS_\mu (W/\R_1)^*$ has two filtrations induced by (\ref{eq:sequence}) with the associated graded $\B_1$ and $\B_2$, respectively. Hence, the cohomology of $\SS_\ll \R_2 \oo \SS_\mu (W/\R_1)^*$ is smaller in general then either the cohomology of  $\B_1$ or  $\B_2$ due to potential cancellations coming from connecting homomorphisms of spectral sequences.

\begin{example}
Consider $X=\Fl(1,2,3)$ and the bundle $\SS_{(3,2)} \R_2 \oo \SS_{(3,1)} (W/\R_1)^*$. We compute the cohomology of the two split bundles, and obtain that the only non-zero spaces are the following:
\[H^2(X,\B_1)  = H^3(X,\B_1)= \SS_{(1,1,-1)} W;\]
\[H^2(X,\B_2)  = H^3(X,\B_2) = \SS_{(1,0,0)} W.\]
Since the cohomology of $\SS_{(3,2)} \R_2 \oo \SS_{(3,1)} (W/\R_1)^*$ is smaller then either of the split bundles $\B_1,\B_2$, this implies that all the potential cancellations above (in degrees $2,3$) must occur, hence
\[H^i \bigl(X, \SS_{(3,2)} \R_2 \oo \SS_{(3,1)} (W/\R_1)^* \bigr)=0, \,\, \mbox{ for all } i\geq 0.\]
\end{example}

\begin{example}
Consider $X=\Fl(1,2,3)$ and the bundle $\SS_{(4,1)} \R_2 \oo \SS_{(4,1)} (W/\R_1)^*$. The only non-zero cohomologies of the split bundles are
\[H^2(X,\B_1)=\SS_{(0,0,0)}W \oplus \SS_{(1,0,-1)} W \oplus \SS_{(2,0,-2)}W \oplus \SS_{(1,1,-2)}W, \,\,\,\,\,\, H^3(X,\B_1)=\SS_{(1,1,-2)}W;\]
\[H^2(X,\B_2)=\SS_{(0,0,0)}W \oplus \SS_{(1,0,-1)} W \oplus \SS_{(2,0,-2)}W \oplus \SS_{(2,-1,-1)}W, \,\,\,\,\,\, H^3(X,\B_2)=\SS_{(2,-1,-1)}W.\]
Hence, all the potential cancellations (in degrees $2,3$) must hold, and the only non-zero cohomology is
\[H^2\bigl(X,\SS_{(4,1)} \R_2 \oo \SS_{(4,1)} (W/\R_1)^*\bigr)=\SS_{(0,0,0)}W \oplus \SS_{(1,0,-1)} W \oplus \SS_{(2,0,-2)}W.\]
\end{example}

In many instances this method is sufficient in describing all the cohomology spaces. However, in general other tools are needed as the following example shows:
\begin{example}\label{ex:later}
Consider $X=\Fl(1,3,4)$ and the bundle $\SS_{(3,1,0)}\R_2 \oo \SS_{(3,1,0)} (W/\R_1)^*$. The only non-zero cohomologies of the split bundles are
\[H^2(X,\B_1)\!=\!H^2(X,\B_2)\!=\SS_{(0,0,0,0)}W \oplus \SS_{(1,1,0,-2)} W \oplus \SS_{(1,1,-1,-1)}W \oplus \SS_{(2,0,0,-2)}W\oplus \SS_{(2,0,-1,-1)}W \oplus \SS_{(1,0,0,-1)}W^{\oplus 3},\] 
\[H^3(X,\B_1)=H^3(X,\B_2)=\SS_{(1,0,0,-1)}W.\]
The representation $\SS_{(1,0,0,-1)}W$ occurs both in degrees $2,3$ for both $\B_1,\B_2$, and we cannot conclude that cancellation holds using only the splitting method. We show in the next section that cancellation indeed holds.
\end{example}

We proceed with a result computing explicitly the cohomology for some hook partitions, which we use in the proof of Theorem \ref{thm:mingens}.

\begin{prop}\label{prop:hook}
Let $0<r_1<r_2<n$ and assume that  $\ll=(a+1, 1^b)$ with $0 < a \leq n-r_2$ and $0\leq b<r_2$. In the following cases, the non-zero cohomology groups of $\SS_\ll \R_2 \oo \SS_\mu (W/\R_1)^*$ are given as follows:
\begin{enumerate}
\item If $a < n-r_2$, then all cohomology groups vanish except for the following, when $b < r_1$:
\[H^{N(-\mu, \ll)}(\SS_\ll \R_2 \oo \SS_\mu (W/\R_1)^*) = S_{\tau(-\mu,\, \ll)} W.\]
\item If $a=n-r_2$, and $\mu = (k)$ with $k\geq 1$, then the following are all the irreducible representations of $\GL(n)$ that are summands of  $H^i(\,\SS_\ll \R_2 \oo \SS_{(k)} (W/\R_1)^*\,)$ (in which case they have multiplicity one): \smallskip
\begin{itemize}
\item[•] $\SS_{(1^a)}W$, \, when $i=a+k-1$, \, $k=b+1$ and $b<r_1$; \medskip
\item[•] $\SS_{(1^{a+b-r_1})}W$, \, when $i=a+r_1$, \, $k=r_1+1$ and $r_1\leq b$; \medskip
\item[•] 
$\SS_{(1^{a+1+j},\,0^{b+r_2-r_1-2j-2}, \,-1^{r_1+j-b},\,r_1-k)}W$, \, for all $j$ with \, $\max\{0, \,b-r_1\}\leq j \leq \min\{b, \, r_2 - r_1 -2\}$, \smallskip \\ when $i=a+r_1$ and $k>r_1$;\medskip
\item[•] $\SS_{(1^{a+j},\,0^{b+r_2-r_1-2j-1}, \,-1^{r_1+j-b},\,r_1-k+1)}W$, \, for all $j$ with \, $\max\{0, \,b-r_1\}\leq j \leq \min\{b, \, r_2 - r_1 -1\}$, \smallskip \\ when $i=a+r_1$ and $k>r_1+1$.
\end{itemize}
\end{enumerate}
\end{prop}

\begin{proof}
For part (1) we use the split bundle $\B_1$. By the Littlewood--Richardson rule, a summand of $\SS_\ll (\R_1 \oplus \R_2/\R_1)$ is of the form $\SS_{(x,1^y)} (\R_2/\R_1) \oo \SS_{(z,1^t)} \R_1$ with $x\leq a+1 \leq n-r_2$. If $x>0$, then using (\ref{eq:push}) we see that $\mathbb{R}^\bullet \pi_* (\SS_{(x,1^y)}(\R_2/\R_1))=0$, so the summand does not give cohomology. Hence, cohomology can occur only for $x=0$, and when we get 
\[H^i(\,\SS_\ll \R_2 \oo \SS_\mu (W/\R_1)^*)=H^{i}(\,\SS_\ll \R_1 \oo \SS_\mu (W/\R_1)^*),\]
for all $i\geq 0$. We conclude by Theorem \ref{thm:bott} again.

For part (2), we first use the split bundle $\B_1$ again. As in the computation above, we see that the only summand of $\SS_\ll (\R_1 \oplus \R_2/\R_1)$ that can yield cohomology is of the form $\SS_{(a+1,1^{j})} \R_2/\R_1 \oo \SS_{(1^{b-j})} \R_1$, for $j$ satisfying $\max\{0, \,b-r_1\}\leq j \leq \min\{b, \, r_2 - r_1 -1\}$. The pushforward as in (\ref{eq:push}) is $\SS_{(1^{a+j+1})} W/\R_1$. By the Pieri rule, we have
\[\SS_{(k)}(W/\R_1)^* \oo \SS_{(1^{a+j+1})} W/\R_1 \cong \SS_{(1^{a+j+1},\,0^{r_2-r_1-j-2},\, -k)}W/\R_1 \oplus \SS_{(1^{a+j},\,0^{r_2-r_1-j-1},\, 1-k)}W/\R_1,\] 
where the first summand is zero for $j=r_2-r_1-1$. The bundle $ \SS_{(1^{a+j+1},\,0^{r_2-r_1-j-2},\, -k)}W/\R_1 \oo \SS_{(1^{b-j})} \R_1$ gives non-zero cohomology if and only if either $b-j=k$ or $k\geq r_1+1$, and in these cases we get the cohomology $\SS_{(1^{a+b-k+1})}W$ and $\SS_{(1^{a+1+j},\,0^{b+r_2-r_1-2j-2}, \,-1^{r_1+j-b},\,r_1-k)}W$ in degrees $a+k$ and $a+r_1$, respectively. Similarly, the bundle $ \SS_{(1^{a+j},\,0^{r_2-r_1-j-1},\, 1-k)}W/\R_1 \oo \SS_{(1^{b-j})} \R_1$ gives non-zero cohomology if and only if either $b-j=k-1$ or $k\geq r_1+2$, and in these cases we get the cohomology $\SS_{(1^{a+b-k+1})}W$ and $\SS_{(1^{a+j},\,0^{b+r_2-r_1-2j-1}, \,-1^{r_1+j-b},\,r_1-k+1)}W$ in degrees $a+k-1$ and $a+r_1$, respectively.

Now replacing back the split bundle $\B_1$ with $\SS_\ll \R_2 \oo \SS_\mu (W/\R_1)^*$, the only potential cancelations between the cohomology obtained above is for the representation $\SS_{(1^{a+b-k+1})}W$ that can appear both in degrees $a+k$ and $a+k-1$. Whenever $\SS_{(1^{a+b-k+1})}W$ appears in degree $a+k$, we must have $k\leq \min\{b,\, r_1\}$, while in degree $a+k-1$ it can appear for $k=b+1$ or $k=r_1+1$. In order to finish the proof, it is enough to show that if $k\leq \min\{b,\, r_1\}$, then the bundle $\SS_\ll \R_2 \oo \SS_{(k)} (W/\R_1)^*$ has no non-zero cohomology groups.

Hence, we can assume $k\leq \min\{b,\, r_1\}$ and consider now the split bundle $\B_2$. Since $k\leq r_1$, by part (1) and Lemma \ref{lem:symm}, the only non-zero cohomology can appear in degree $N(0^{a-1},\, -k,\, a+1,\, 1^b, \, 0^{r_2-b-1})$. But since $k\leq b$, we have $N(0^{a-1},\, -k,\, a+1,\, 1^b, \, 0^{r_2-b-1})=-\infty$, thus yielding the claim.
\end{proof}

All cohomology above is concetrated in a single degree, and experiments show that this happens frequently (see also Proposition \ref{prop:lowcases}). Nevertheless, it does not happen always, as illustrated by the following example.

\begin{example}\label{ex:notsingle}
Consider the bundle $\SS_{(4,4)} \R_2 \oo \SS_{(2,0)} (W/\R_1)^*$ on $\Fl(1,2,3)$. Then the following are the only nonvanishing cohomology groups
\[H^2(\SS_{(4,4)} \R_2 \oo \SS_{(2,0)} (W/\R_1)^*)=\SS_{(3,3,0)} W, \quad H^3(\SS_{(4,4)} \R_2 \oo \SS_{(2,0)} (W/\R_1)^*)=\SS_{(2,2,2)} W.\]
\end{example}

\subsection{Refinements via Schur complexes}\label{sec:refineschur}

Consider an exact sequence $0 \to A \to B \to C\to 0$ of vector spaces (or locally free sheaves). Then for any partition $\ll$, the following is a right resolution of the module $\SS_\ll A$:

\begin{equation}\label{eq:schur1}
\SS_\ll B \to \bigoplus_{\substack{|\mu|=|\ll|-1 \\ |\nu|=1}} (\SS_{\mu} B \oo \SS_{\nu'} C)^{\oplus c^\ll_{\mu,\nu}}  \to \dots \to \bigoplus_{\substack{|\mu|=|\ll|-i \\ |\nu|=i}} (\SS_{\mu} B \oo \SS_{\nu'} C)^{\oplus c^\ll_{\mu,\nu}} \to \dots \to \SS_{\ll'} C \to 0.
\end{equation}
Here $c^\ll_{\mu,\nu}$ denotes the Littlewood--Richardson coefficient corresponding to the partitions $\ll,\mu,\nu$. The maps in the complex can be constructed explicitly (see \cite{schur} or \cite[Section 2.4]{jerzy} for more about Schur complexes).


The idea is to consider the complex (\ref{eq:schur1}) for either one of the following exact sequences:
\[0 \to \R_2 \to W \to W/\R_2\to 0, \mbox{ resp. } 0 \to (W/\R_1)^* \to W^* \to \R_1^* \to 0. \]
Let us explain this in the first case. Considering the respective Schur complex (\ref{eq:schur1}), we get a right resolution of the module $\SS_{\ll} \R_2$. The syzygies of this complex can be analyzed directly in many situations, giving complementary  information from the one obtained by the splitting method in Section \ref{sec:split}.

\begin{example}\label{ex:latnow}
Here we finish Example \ref{ex:later} to show that all the cancelations hold indeed, so that the cohomology of $\SS_{(3,1,0)}\R_2 \oo \SS_{(3,1,0)} (W/\R_1)^*$ on $\Fl(1,3,4)$ is concentrated in degree $2$. Consider the Schur complex as the resolution of $\SS_{(3,1,0)}\R_2$:
\[0\to \SS_{(3,1,0)}\R_2 \to \SS_{(3,1,0,0)} W \to (\SS_{(2,1,0,0)} W \oplus \SS_{(3,0,0,0)} W) \oo (W/\R_2) \to \SS_{(2,0,0,0)} W \oo \SS_{(2)} (W/\R_2) \to 0.\]
Let us analyze the middle syzygy $K$ of the complex above. The end of the sequence can be interpreted as
\[0\to K \to \SS_{(2,0,0,0)} W \oo W \oo (W/\R_2) \to \SS_{(2,0,0,0)} W \oo (W/\R_2) \oo (W/\R_2)\to 0,\]
where the last map is induced from the projection $W\to W/\R_2$. Hence, $K\cong \SS_{(2,0,0,0)}W \oo \R_2 \oo (W/\R_2)$ and we have an exact sequence
\begin{multline}\label{eq:long}
0 \to \SS_{(3,1,0)}\R_2 \oo \SS_{(3,1,0)} (W/\R_1)^* \to \SS_{(3,1,0,0)} W \oo \SS_{(3,1,0)} (W/\R_1)^* \to \\
\to \SS_{(2,0,0,0)}W \oo \R_2 \oo W/\R_2 \oo \SS_{(3,1,0)} (W/\R_1)^* \to 0.
\end{multline}
Now we take the induced the long exact sequence of cohomology, and see that the second and third bundles (by replacing the bundle $\R_2$ with the split bundle $\R_1 \oplus (\R_2/\R_1)$) above give cohomology only in degree $1$, hence $\SS_{(3,1,0)}\R_2 \oo \SS_{(3,1,0)} (W/\R_1)^*$ has no cohomology in degree $>2$.
\end{example}

We can tensor the right resolution of $\SS_{\ll} \R_2$ above by $\SS_\mu (W/\R_1)^*$, to obtain a right resolution of $\SS_\ll \R_2 \oo \SS_\mu (W/\R_1)^*$. Note that the terms of the resolution involve only Schur functors of the bundles $W, (W/\R_1)^*$ and $W/\R_2$. Hence, as in (\ref{eq:proj}) by the projection formula and the (relative) Borel-Weil-Bott theorem, we see that the resolution is $\pi_*$-acyclic and we obtain the following complex on $\Gr(r_1,n)$, tensored by $\SS_\mu (W/\R_1)^*$:
\begin{equation}\label{eq:trunc}
\SS_\ll W \to \!\!\!\!\!\!\!\!\!\! \bigoplus_{\substack{|\mu|=|\ll|-1 \\ |\nu|=1,\, \nu_1 \leq n-r_2}}\!\!\!\!\!\!\!\!\!\! (\SS_{\mu} W \oo \SS_{\nu'} (W/\R_1))^{\oplus c^\ll_{\mu,\nu}}  \to \dots \to \!\!\!\!\!\!\!\!\!\! \bigoplus_{\substack{|\mu|=|\ll|-i \\ |\nu|=i, \, \nu_1 \leq n-r_2}} \!\!\!\!\!\!\!\!\!\! (\SS_{\mu} W \oo \SS_{\nu'} (W/\R_1))^{\oplus c^\ll_{\mu,\nu}} \to \dots \to \SS_{\ll'} (W/\R_1) \to 0. \vspace{-0.05in}
\end{equation}

This is a truncated Schur complex as it appears in \cite[Chapter 2, Exercise 22]{jerzy}. Taking hypercohomology of the complex above tensored by $\SS_\mu (W/\R_1)^*$ yields a spectral sequence converging to the cohomology of the bundle $\SS_\ll \R_2 \oo \SS_\mu (W/\R_1)^*$. This can be harnessed effectively for gaining additional information.

\begin{example}\label{ex:truncschur}
Consider the bundle $\SS_{(4,1,0,0)}\R_2 \oo \SS_{(4,1,0,0)} (W/\R_1)^*$ on $\Fl(2,4,6)$. Using either split bundle
\begin{multline*} 
H^4(\B_i) = \SS_{(2,1,0,0,-1,-2)}W\oplus\SS_{(2,1,0,-1,-1,-1)}W\oplus\SS_{(2,0,0,0,-1,-1)}W\oplus\SS_{(1,1,1,0,-1,-2)}W\oplus\SS_{(1,1,1,-1,-1,-1)}W\oplus \\ \oplus\SS_{(1,1,0,0,0,-2)}W\oplus  (\SS_{(1,1,0,0,-1,-1)}W)^{\oplus 3} \oplus \SS_{(1,0,0,0,0,-1)}W,\quad  \mbox{ and } \quad H^5(\B_i) = \SS_{(1,1,0,0,-1,-1)}W.
\end{multline*}
We show using the complex (\ref{eq:trunc}) (tensored with $\SS_{(4,1,0,0)} (W/\R_1)^*$)  that the representation $\SS_{(1,1,0,0,-1,-1)}W$ in degree $5$ cancels out with a copy in degree $4$ (in particular, the cohomology is concentrated in degree $4$).  
By inspection, the last term in the complex is the only term that can give cohomology $\SS_{(1,1,0,0,-1,-1)}W$
 in degree $5$ through the spectral sequence. Hence, it is enough to see that the last map in the complex induces a surjective map on cohomology in degree $2$ on the level of the corresponding isotypic component. We show that this happens already when we restrict the last map in the complex to the summands (tensored with $\SS_{(4,1,0,0)} (W/\R_1)^*$)
 \[(\SS_{(3)} W \oplus \SS_{(2,1)}W) \oo \SS_{(1,1)} (W/\R_2) \to \SS_{(2,0)}W \oo \SS_{(2,1)}(W/\R_2).\]
Like in the previous example, this map can be interpreted in the following short exact sequence (for example, using Lemma \ref{lem:cat} (4) and the fact that on the level of the two summands the maps are non-zero)
 \[0\to \SS_{(2,0)}W \oo \R_2 \oo \SS_{(1,1)} (W/\R_2) \to \SS_{(2,0)}W \oo W \oo \SS_{(1,1)} (W/\R_2) \to \SS_{(2,0)}W \oo (W/\R_2)\oo \SS_{(1,1)}(W/\R_2) \to 0.\]
Replacing $\R_2$ with $\R_1 \oplus (\R_2/\R_1)$, it follows easily that $H^3(\SS_{(2,0)}W \oo \R_2 \oo \SS_{(1,1)} (W/\R_2) \oo \SS_{(4,1,0,0)} (W/\R_1)^*)=0$.
\end{example}

By a case-by-case analysis (using also Lemma \ref{lem:symm} and Proposition \ref{prop:hook}), we see that the splitting method together with refinements as above yields the following result.
\begin{prop}\label{prop:lowcases}
When $n\leq 4$ and $\ll_1,\mu_1\leq 3$, the cohomology groups $\SS_\ll \R_2 \oo \SS_\mu (W/\R_1)^*$ on $\Fl(r_1,r_2,n)$ are concentrated in a single degree, and a fortiori computable using the splitting method.
\end{prop}

In principle, the Schur complexes above can be tracked in the category of homogeneous vector bundles (or $\GL(n)$-bundles) over the Grassmannian. The latter can be described using a quiver with relations, with vertices in bijection with the simple objects $\SS_\a \R \oo \SS_\b \Q$ \cite{bonkap, hille, ottaviani}. Although in general the method can get quite complicated, it gives good control of the first page of the spectral sequence associated to the truncated Schur complex, which we illustrate by an example.

First, we gather some convenient facts about the bundles that appear in our complexes. Recall the notions of radical and socle filtrations (e.g. \cite[Section V.1]{elements}). When these filtrations coincide, the corresponding object is sometimes called rigid in the literature. This is the case for the bundles $\M_{\mu,\nu}:= \SS_{\mu} W \oo \SS_{\nu} \Q$.

\begin{lemma}\label{lem:cat} The following statements hold in the category of $\GL(n)$-bundles over $\Gr(r,n)$:
\begin{enumerate}
\item $\M_{\mu,\nu}$ has an ascending filtration $F^\bullet \M_{\mu,\nu}$ with 
\[\Gr^i \M_{\mu,\nu} := F^i\M_{\mu,\nu}/F^{i-1}\M_{\mu,\nu} \cong \displaystyle\bigoplus_{\substack{\alpha,\beta\\ |\alpha| = |\mu|-i}} (\SS_{\alpha} \R \oo \SS_{\beta} \Q  \oo \SS_{\nu} \Q)^{\oplus c^{\mu}_{\alpha,\beta}}.\]
\vspace{-0.13in}
\item The radical and socle filtrations of $\M_{\mu,\nu}$ agree with $F^\bullet \M_{\mu,\nu}$ (up to a shift).
\item $\End(\M_{\mu,\nu})=\kk$.  In particular, $\M_{\mu,\nu}$ is indecomposable.
\item Let $\a, \b$ partitions with $|\a| = |\mu|-1$ and $|\b|=|\nu|+1$. Then $\dim_\kk \Hom(\M_{\mu,\nu}, \,\M_{\alpha, \beta})\leq 1$ with equality if and only if $\a\subset \mu$ and $\nu\subset \b$, in which case the non-zero map is (up to scaling) the composition
\[\M_{\mu,\nu} \hookrightarrow \M_{\alpha, \nu} \oo W \twoheadrightarrow  \M_{\alpha, \nu} \oo \Q \twoheadrightarrow \M_{\alpha, \beta},\]
where the first and last maps are given by the Pieri rule (e.g. \cite[Corollary 2.3.5]{jerzy}).
\end{enumerate}
\end{lemma}

\begin{proof}
The claims follow by the repeated application of the Littlewood--Richardson rule, the Borel--Weil--Bott theorem and the natural isomorphism
\begin{equation}\label{eq:ext}
\Ext^i(\L_1, \L_2) \cong H^i(\L_1^*\oo \L_2)^{\GL(n)}, \qquad \mbox{ for all } i\geq 0,
\end{equation}
where $\L_1, \L_2$ are arbitrary $\GL(n)$-bundles over $\Gr(r,n)$. The map in part (4) is a variant of the Olver map (e.g. see \cite[Section 8]{ottaviani}.
\end{proof}

The combination of all the methods discussed yields a potent weapon for tackling more complicated cases.

\begin{example}\label{ex:quiver}
Consider the bundle $\SS_{(4,1,1,0)}\R_2 \oo \SS_{(4,1,1,0)} (W/\R_1)^*$ on $\Fl(2,4,6)$. Using either split bundle, we obtain (compare with Example \ref{ex:later})
\[H^4(\B_i) = \SS_{(2,0,0,0,0,-2)}W\oplus \SS_{(2,0,0,0,-1,-1)}W\oplus\SS_{(1,1,0,0,0,-2)}W\oplus\SS_{(1,1,0,0,-1,-1)}W\oplus  (\SS_{(1,0,0,0,0,-1)}W)^{\oplus 3} \oplus \SS_{(0^6)}W,\]
\[H^5(\B_i) = \SS_{(1,0,0,0,0,-1)}W.\]
We show that the representation $\chi:=\SS_{(1,0,0,0,0,-1)}W$ in degree $5$ cancels out (so cohomology is concentrated in degree $4$).  
By inspection, a single summand in the third term of the complex (\ref{eq:trunc}) is the only one that can give cohomology $\chi$
 in degree $5$ through the spectral sequence. Consider the following non-zero map (see Lemma \ref{lem:cat} (4)) from the complex with target this summand (tensored with $\SS_{(4,1,1)} \Q^*$, where $\R = \R_1$)
 \[f: \quad \SS_{(2,1,1)} W \oo \SS_{(1,1)} \Q \longrightarrow \SS_{(2,1)}W \oo  \SS_{(2,1)} \Q.\]
Denote by $H^i_\chi(-)$ the isotypic component corresponding to $\chi$ of a cohomology group. The only non-vanishing such components of the two bundles in question are in degree $2$ with $\dim_\kk H^2_\chi(\M_{(2,1,1) , (1,1)} \oo \SS_{(4,1,1)} \Q^* ) =4$ and $\dim_\kk H^2_\chi(\M_{(2,1) , (2,1)}\oo \SS_{(4,1,1)} \Q^* ) = 2$. It is enough to show that $f$ induces a surjective map between these spaces. We consider the filtrations $F^\bullet \M_{(2,1,1) , (1,1)}$ and $F^\bullet \M_{(2,1) , (2,1)}$ as in Lemma \ref{lem:cat} (1). The only simples in the associated graded bundles that (tensored with $\SS_{(4,1,1)} \Q^*$) give cohomology $\chi$ are 
\[\SS_{(2,1)}\R \oo \SS_{(2,1)}\Q, \, \SS_{(2,1)}\R \oo \SS_{(1,1,1)}\Q, \, \SS_{(1,1)}\R \oo \SS_{(3,1)}\Q, \, \SS_{(1,1)}\R \oo \SS_{(2,2)}\Q, \, \SS_{(1,1)}\R \oo \SS_{(2,1,1)}\Q, \, \SS_{(1,1)}\R \oo \SS_{(1,1,1,1)}\Q.\]
The respective multiplicities of $\chi$ are $2,1,1,1,2,1$ in degrees $3,3,2,2,2,2$. The multiplicities of these bundles in $\M_{(2,1,1) , (1,1)}$ and $\M_{(2,1) , (2,1)}$ are $1,1,1,1,2,1$ and $1,0,1,1,1,0$, respectively. The following is the relevant part of the quiver of the category of $\GL(6)$-bundles, where $x_\a$ stands for a bundle $\SS_{\beta} \R \oo \SS_\a \Q$ as above:
\vspace{-0.1in}
\[\xymatrix@R-1pc{ x_{(3,1)}  \ar[dr] & x_{(2,2)} \ar[d] & x_{(2,1,1)} \ar[ld]  \ar[d] & x_{(1,1,1,1)} \ar[dl] \\
& x_{(2,1)}  & x_{(1,1,1)} & }\vspace{-0.1in}\]
The bundles $\M_{(2,1,1) , (1,1)}$ and $\M_{(2,1) , (2,1)}$ can be displayed in this part of the quiver as follows:
\vspace{-0.1in}
\[\vcenter{\xymatrix@R-0.5pc@C+0.7pc{ 
\kk  \ar[dr]_{a_1} & \kk \ar[d]^{a_2} & \kk^2 \ar[ld]^{a_3}  \ar[d]^{a_4} & \kk \ar[dl]^{a_5} \\
& \kk  & \kk & 
}} \qquad \mbox{ and } \qquad
\vcenter{\xymatrix@R-0.5pc@C+0.7pc{ 
\kk  \ar[dr]_{b_1} & \kk \ar[d]^{b_2} &\kk \ar[ld]^{b_3}  \ar[d] & 0 \ar[dl] \\
& \kk  & 0 & 
}}\vspace{-0.1in}\]
By Lemma \ref{lem:cat} (2) all the maps $a_i,b_i$ are nonzero. We see (e.g. using Lemma \ref{lem:cat} (4)) that the map induced at $x_{(2,1)}$ by $f$ is nonzero. Thus, $f$ induces surjective maps at each of the $x_\a$ as above (note $f$ preserves the filtrations by Lemma \ref{lem:cat}(2)). This shows that in the following diagram with exact rows, the maps $b,c$ are onto
\vspace{-0.12in}
\[
\xymatrix@C-1.2pc@R-0.6pc{0 \! \to \! H^2_\chi(F^1\!\M_{(2,1,1),(1,1)} \!\oo \! \SS_{(4,1,1)}\Q^*) \ar[r] \ar[d]^{a} & H^2_\chi(\Gr^1\!\! \M_{(2,1,1),(1,1)} \!\oo \! \SS_{(4,1,1)}\Q^*) \ar[r] \ar[d]^b & H^3_\chi(F^0\!\M_{(2,1,1),(1,1)} \!\oo \! \SS_{(4,1,1)}\Q^*) \ar[d]^{c} \! \to \! 0\\
0 \! \to \! H^2_\chi(F^1\!\M_{(2,1),(2,1)} \!\oo \SS_{(4,1,1)}\Q^*) \ar[r] & H^2_\chi(\Gr^1\!\! \M_{(2,1),(2,1)}\!\oo \SS_{(4,1,1)}\Q^*) \ar[r] & H^3_\chi(F^0\!\M_{(2,1),(2,1)} \!\oo \SS_{(4,1,1)}\Q^*) \! \to \! 0 }\vspace{-0.02in}
\]
We conclude that $a$ is onto by showing that $\ker b \to \ker c$ is so. It is enough to see the surjectivity of the map $H^2_\chi(\SS_{(1,1)}\R \oo \SS_{(1,1,1,1)}\Q\oo \SS_{(4,1,1}\Q^*) \to H^3_\chi(\SS_{(2,1)}\R \oo \SS_{(1,1,1)}\Q \oo \SS_{(4,1,1}\Q^*)$, or equivalently of the map $H^2(\SS_{(1,1)}\R \oo \SS_{(1,0,0,-3)}\Q) \to H^3(\SS_{(2,1)}\R \oo \SS_{(1,0,-1,-3)}\Q)$. A homological argument using (\ref{eq:ext}) shows that the non-trivial extension given by $a_5$ induces a non-trivial extension between $\SS_{(1,1)}\R \oo \SS_{(1,0,0,-3)}\Q$ and $\SS_{(2,1)}\R \oo \SS_{(1,0,-1,-3)}\Q$. The surjectivity of the required connecting map now follows by \cite[Theorem 6.6]{ottaviani}. 
\end{example}

\subsection{A definitive algorithm via Schur complexes}\label{sec:definschur}

For larger cases, the methods above become cumbersome to use in practice. In this section, we outline an algorithm to compute the cohomology groups of $\SS_\ll \R_2 \oo \SS_\mu (W/\R_1)^*$ on $\Fl(r_1,\,r_2,\,n)$ in general, reducing the problem to elementary linear algebra. This is done by constructing an explicit acyclic resolution of this bundle as follows.

Consider the resolution of both $\SS_\ll \R_2$ and $\SS_\mu (W/\R_1)^*$ by the respective Schur complexes (\ref{eq:schur1}). Tensoring the two complexes yields a double complex. Taking the total complex of this double complex, gives a resolution $Tot^\bullet$ of $\SS_\ll \R_2 \oo \SS_\mu (W/\R_1)^*$. 

\begin{prop}\label{prop:acyclic}
The complex $Tot^\bullet$ is an acyclic resolution of $\SS_\ll \R_2 \oo \SS_\mu (W/\R_1)^*$.
\end{prop}

\begin{proof}
The terms of the complex $Tot^\bullet$ are all of the form $\SS_{\a}(W/\R_2) \oo \SS_{\b}\R_1^* \oo \SS_{\gamma} W$ for partitions $\a,\b$, and a non-increasing sequence $\gamma$. By the Borel--Weil--Bott theorem, we have
\begin{equation}\label{eq:cartan}
H^0(\,\SS_{\a}(W/\R_2) \oo \SS_{\b}\R_1^*\,) \cong \SS_{(\a, \, 0^{(r_2-r_1)},\,-\b)}W,
\end{equation}
and $H^i(\,\SS_{\a}(W/\R_2) \oo \SS_{\b}\R_1^*\,)=0$, for $i>0$. In particular, this shows that $Tot^\bullet$ is an acyclic resolution of the bundle $\SS_\ll \R_2 \oo \SS_\mu (W/\R_1)^*$, as required. 
\end{proof}

Now $H^0(Tot^\bullet)$ gives a complex of vector spaces with explicit terms (\ref{eq:cartan}) (tensored with $\SS_\gamma W$). The maps in $H^0(Tot^\bullet)$ are explicit as well: they can be obtained from the tensor product of maps of (truncated) Schur complexes, restricted to the Cartan piece (\ref{eq:cartan}). This follows since the constructions are functorial and the maps involved are induced by taking sections in commutative diagrams of the following type
\[\xymatrix@R-0.8pc{ \SS_{\nu}W \oo \SS_{\a} W \oo \SS_{\b_1} W^* \oo \SS_{\nu_1}W^* \ar[r]\ar[d] & \ar[d] \SS_{\nu}W \oo \SS_{\a} W \oo \SS_{\b_2} W^* \oo \SS_{\nu_2}W^* \\
\SS_{\nu}W \oo \SS_{\a} (W/\R_2) \oo \SS_{\b_1} \R_1^* \oo \SS_{\nu_1}W^* \ar[r] & \SS_{\nu}W \oo \SS_{\a} (W/\R_2) \oo \SS_{\b_2} \R_1^* \oo \SS_{\nu_2}W^*}\]

Once the differentials of $H^0(Tot^\bullet)$ are realized as above, we can recover the cohomology as a representation of $\GL(n)$ as follows. Pick $T\subset \GL(n)$ to be the subgroup of diagonal matrices, and choose $T$-weight bases in the terms of our complex. A natural choice of a basis is formed by pairs of standard $\mathbb{Z}_2$-graded tableau (see \cite[Section 2.4]{jerzy})), which are in turn "standardized" to the Cartan piece (\ref{eq:cartan}). Then computing the cohomology of $H^0(Tot^\bullet)$ using this basis gives the $T$-character of the cohomology groups of $\SS_\ll \R_2 \oo \SS_\mu (W/\R_1)^*$, which in turn determines its $\GL(n)$-structure uniquely. Alternatively, since the differentials are $\GL(n)$-equivariant, it is enough to keep track of highest weight vectors only.

We note that Schur complexes have been recently implemented by \cite{schuralg} through the computer algebra system Macaulay2 \cite{M2}.

\section{Applications for the equioriented $\A_3$ quiver}\label{sec:applications}

We start with the determination of the minimal free resolutions of orbit closures $\ol{O}_V$ for the $\A_3$ quiver. Recall the complex $F_\bullet$ with the notation as in (\ref{eq:basic}).

\begin{theorem}\label{thm:rational} 
The variety $\ol{O}_V$ is normal with rational singularities, and $F_\bullet$ is the minimal free resolution of the defining ideal of $\overline{O}_V$.
\end{theorem}

\begin{proof}
We apply  \cite[Theorem 5.1.3]{jerzy} as described in Section \ref{sec:geom}. We claim that in (\ref{eq:basic}) we have $F_0=A$ and $F_i =0$ for $i<0$. By (\ref{eq:mult}) and (\ref{eq:durf}), we are left to show that
\begin{equation}\label{eq:vanish}
H^i\big(\SS_{\ll} \R_2 \oo \SS_{\mu}(W/\R_1)^*\big) =0,
\end{equation}
whenever $i\geq |\ll| + |\mu|- c\cdot u$, $\ll_{u+c}\geq u$ and at least one of $\ll$ or $\mu$ is non-zero (here $u=u_\ll$). When one of $\ll$ or $\mu$ is zero, the claim follows as in (\ref{eq:durf}). So we can assume that both $\ll$ and $\mu$ are non-zero.

First, note that the Schur complex (\ref{eq:schur1}) resolving $\SS_{\mu'}(W/\R_1)^*$) has length at most $|\mu|$. Hence, so does the truncated complex obtained analogously to (\ref{eq:trunc}) (tensored with $\SS_\ll \R_2$). The bundles appearing in this complex are all of the form $S_\nu W^* \oo S_\gamma \R_2$, for $\nu$ a partition and $\gamma$ a non-increasing sequence with Durfee size at most $u$ satisfying $\gamma_1+ \dots +\gamma_{u} \leq \ll_1+ \dots + \ll_{u}$. Hence, it suffices to show that 
\[H^i(S_\gamma \R_2)=0, \mbox{ for } i\geq |\ll|-c\cdot u.\]
But $|\ll|-c\cdot u \geq \gamma_1+ \dots +\gamma_{u}$, hence we conclude similarly to (\ref{eq:vanish}) in the case $\mu=0$.
\end{proof}

For $1$-step orbit closures the above result follows from \cite{bubfree}. The following example is not $1$-step.

\begin{example}
Consider the variety of pairs of matrices $(A,B)$, where $A$ is $4\times 3$ and $B$ is $3\times 4$, so the dimension vector is $\mathbf{d}=(3,4,3)$. Consider the orbit closure $\ol{O}_V$ given by $\rank A\leq 1, \rank B\leq 1$ and $BA=0$, so that $V=I_2\oplus P_2 \oplus S$, with $S$ a semi-simple representation of $\A_3$. Then we see that $\ol{O}_V$ has codimension 13, and the only representation in $F_{13}$ is
\[\SS_{(3,3,3)} V_1 \oo \SS_{(0,0,0,0)} V_2 \oo \SS_{(3,3,3)} V_3^*.\]
$\ol{O}_V$ is Gorenstein (therefore has a symmetric resolution). By Proposition \ref{prop:lowcases}, calculating the cohomology groups of bundles can be done using the splitting method described in Section \ref{sec:split}. The number of irreducible $\GL(\mathbf{d})$-module summands obtained in the terms $F_0,F_1,F_2,F_3,F_4,F_5,F_6$ are $1,3,6,17,35,48,52$, respectively. In particular, the number of irreducible $\GL(\mathbf{d})$-modules in the terms of $F_\bullet$ form a unimodal sequence. We note that the fact that $\ol{O}_V$ has rational singularities follows in this case also from \cite[Corollary 4.2]{radical}.
\end{example}

We write $\kk[X,Y]=\kk[\Rep(Q,\mathbf{d})]$ for the coordinate ring of the space of matrices $(A,B)$, with $X$ and $Y$ being the corresponding matrices of generic variables. 

\begin{theorem}\label{thm:mingens}
Let $\ol{O}_V$ be the orbit closure given by matrices $(A,B)$ with $\rank A \leq a, \, \rank B \leq b$ and $\rank BA \leq c$. Then the minimal generators of the defining ideal of $\ol{O}_V$ in $\kk[X,Y]$ are given by the $(a+1)\times (a+1)$ minors of $X$, $(b+1) \times (b+1)$-minors of $Y$, together with the $(c+1) \times (c+1)$ minors of $Y\cdot X$ when $c<\min\{a,b\}$.
\end{theorem}

\begin{proof}
Put $n=d_2, r_1=a, r_2 = n-b+c$ as in Section \ref{sec:geom}. We can assume $0<r_1<r_2<n$. Continuing with the reasoning as in Theorem \ref{thm:rational}, the term $F_1$ of the complex $F_\bullet$ is built from the cohomology groups
\begin{equation}\label{eq:cohomin}
H^{|\ll|+|\mu|-c\cdot u - 1}\big(\,\SS_{\ll} \R_2 \oo \SS_{\mu}(W/\R_1)^*\,\big),
\end{equation}
where at least one of $\ll$ or $\mu$ is non-zero, and $\ll_{u+c}\geq u$ (with $u=u_\ll$). 

First, assume that $\ll=0$. Through a computation similar to (\ref{eq:durf}), we obtain that $H^{|\mu|-1}(\SS_{\mu}(W/\R_1)^*)\neq 0$ if and only if $\mu=(a+1)$, in which case $H^{a}(\SS_{(a+1)}(W/\R_1)^*) = \SS_{(1^{a+1})} W^*$. By (\ref{eq:mult}), the contributing term to $F_1$ is the representation $\SS_{(1^{a+1})} V_1 \oo \SS_{(1^{a+1})} V_2^*$. This $\GL(\mathbf{d})$-representation appears in $\kk[X,Y]$ with multiplicity one, and it is spanned by the $(a+1) \times (a+1)$ minors of $X$.

Now let $u\geq 1$. Consider the truncated complex with terms of the form $\SS_\nu W^* \oo S_\gamma \R_2$  as in the proof of Theorem \ref{thm:rational}. The latter proof shows that if the group (\ref{eq:cohomin}) is not zero, then it must correspond to the cohomology of a bundle from the last term of the complex. Hence, there is a non-increasing sequence $\gamma$ in the Littlewood--Richardson product of $\ll$ and $-\mu'$ with  $H^{|\ll|-c\cdot u - 1} (\SS_\gamma \R_2) \neq 0$. 

If $\gamma_1\leq 0$, then by Theorem \ref{thm:bott} we get that $|\ll|=c\cdot u + 1$, hence $\ll=(1^{c+1})$. By Proposition \ref{prop:hook} (1), the bundle $\SS_{\ll} \R_2 \oo \SS_{\mu}(W/\R_1)^*$ can have cohomology only in degree $N(-\mu, \ll)$, when $c<a$. By \cite[Lemma 3.2]{bubfree}, we have 
\[N(-\mu, \ll) \leq u_\mu + |\mu^+| \leq  u_\mu^2 + |\mu^+|+ |\mu^-| = |\mu|.\]
Hence, in order for equalities to hold above, we must have $\mu^-=0$ and $u_\mu^2 = u_\mu$. The case $u_\mu=0$ yields no cohomology in (\ref{eq:cohomin}), hence we must have $\mu = (k)$, for some $k\geq 1$. An easy computation now shows that $N(-\mu, \ll)=|\mu|$ if and only if $k=c+1$, when
\[H^{c+1}(\,\SS_{(1^{c+1})} \R_2 \oo \SS_{(c+1)}(W/\R_1)^*\,) = \SS_{(0^n)} W.\]
By (\ref{eq:mult}), the contributing term to $F_1$ is the representation $\SS_{(1^{c+1})} V_1  \oo \SS_{(1^{c+1})} V_3^*$, which is spanned by the $(c+1)\times (c+1)$ minors of $Y\cdot X$.

Now assume that $\gamma_1\geq 1$, so that $u_\gamma \geq 1$. By Theorem \ref{thm:bott}, $H^i(\SS_\gamma \R_2)$ is non-zero if and only if $i=u_\gamma \cdot (b-c)$ and $\gamma_{u_\gamma}\geq u_\gamma + b-c$. We have 
\[|\ll|-c\cdot u - 1 \geq u \cdot \gamma_{u_\gamma} -1 \geq u\cdot u_\gamma-1+ u \cdot (b-c) \geq  u_\gamma \cdot (b-c).\]
For the equalities to hold, we must have $u=u_\gamma=1$, moreover, $\ll$ must be the hook $\ll = (b-c+1, 1^c)$. If $\mu=0$, then by (\ref{eq:mult}) we obtain the representation $\SS_{(1^{b+1})} V_2 \oo \SS_{(1^{b+1})} V_3^*$ that is spanned by the $(b+1)\times (b+1)$ minors of $Y$.

We are left to show that if $\ll=(b-c+1,\, 1^c)$ and $u_\mu \geq 1$, then the cohomology (\ref{eq:cohomin}) is zero. Consider the split bundle $\B_1$ as in Section \ref{sec:split}. As seen in the proof of Proposition \ref{prop:hook} (2), those summands in the decomposition of $\SS_\ll (\R_1 \oplus \R_2/\R_1) $ that can yield cohomology must of the form $\SS_\ll \R_1$, or $\SS_{(b-c+1,\, 1^x)}(\R_2/R_1) \oo \SS_{(1^y)} \R_1$ with $x+y=c$. In the former case, the cohomology is in degree $N(-\mu,\,\ll)$. By \cite[Lemma 3.2]{bubfree}, we have
\[N(-\mu,\,\ll)\leq u_\mu + |\mu^+|+ b-c \leq u_\mu^2 + |\mu^+|+ |\mu^-| +b-c = |\mu| + b-c.\]
In order for equalities to hold in the above, we must have $u_\mu=1$ and $\mu^-=0$. By Proposition \ref{prop:hook} (2), the corresponding cohomology (\ref{eq:cohomin}) vanishes in this case. On the other hand, working with the summand $\SS_{(b-c+1, \,1^x)}(\R_2/R_1) \oo \SS_{(1^y)} \R_1$ we arrive to the cohomology of $\SS_{(1^y)} \R_1 \oo \SS_\gamma (W/\R_1)$, where $\gamma$ is in the Littlewood--Richardson product of $-\mu$ and $(1^{x+b-c+1})$. Write $\gamma=(\alpha,\,-\beta)$, with $\alpha, \beta$ partitions. By \cite[Lemma 3.2]{bubfree}, we have
\[N(\gamma, \,1^y)= N(-\beta,\, 1^y) \leq u_\beta + |\beta^+| \leq  u_\mu^2 + |\mu^+| \leq |\mu|.\]
In order for equalities to hold, we again must have $u_\mu=1$ and $\mu^-=0$. By Proposition \ref{prop:hook} (2) again, the cohomology (\ref{eq:cohomin}) vanishes in this last case.
\end{proof}

We note that the fact that the minors generate a radical ideal follows also from \cite{lak}.

\section*{Acknowledgments}

The author is grateful to Jerzy Weyman for helpful conversations and suggestions.

\bibliographystyle{alpha}
\bibliography{biblo}

\newcommand{\etalchar}[1]{$^{#1}$}
\begin{thebibliography}{BHL{\etalchar{+}}19}

\bibitem[ABW82]{schur}
K.~Akin, D.~A. Buchsbaum, and J.~Weyman.
\newblock Schur functors and {S}chur complexes.
\newblock {\em Adv. in Math.}, 44(3):207--278, 1982.

\bibitem[AF85]{abeasis}
S.~Abeasis and A.~Del Fra.
\newblock Degenerations for the representations of a quiver of type
  {$\mathcal{A}_m$}.
\newblock {\em J. Algebra}, 93:376--412, 1985.

\bibitem[AFK81]{abea}
S.~Abeasis, A.~Del Fra, and H.~Kraft.
\newblock The geometry of representations of {$A_m$}.
\newblock {\em Mathematische Annalen}, 256(3):401--418, 1981.

\bibitem[ASS06]{elements}
I.~Assem, D.~Simson, and A.~Skowro{\'n}ski.
\newblock {\em Elements of the representation theory of associative algebras,
  {V}ol. 1, {T}echniques of representation theory}, volume~65 of {\em London
  Mathematical Society Student Texts}.
\newblock Cambridge University Press, Cambridge, 2006.

\bibitem[BHL{\etalchar{+}}19]{schuralg}
Michael~K. Brown, Hang Huang, Robert~P. Laudone, Michael Perlman, Claudiu
  Raicu, Steven~V. Sam, and Jo\~{a}o Santos.
\newblock Computing {S}chur complexes.
\newblock {\em J. Softw. Algebra Geom.}, 9(2):111--119, 2019.

\bibitem[BK90]{bonkap}
A.~I. Bondal and M.~M. Kapranov.
\newblock Homogeneous bundles.
\newblock In {\em Helices and vector bundles}, volume 148 of {\em London Math.
  Soc. Lecture Note Ser.}, pages 45--55. Cambridge Univ. Press, Cambridge,
  1990.

\bibitem[BZ01]{orb1}
G.~Bobi\'{n}ski and G.~Zwara.
\newblock Normality of orbit closures for {D}ynkin quivers of type
  {$\mathbb{A}_n$}.
\newblock {\em Manuscripta Math.}, 105:103--109, 2001.

\bibitem[BZ02]{orb2}
G.~Bobi\'{n}ski and G.~Zwara.
\newblock Schubert varieties and representations of {D}ynkin quivers.
\newblock {\em Colloq. Math.}, 94:285--309, 2002.

\bibitem[GS]{M2}
D.~R. Grayson and M.~E. Stillman.
\newblock Macaulay2, a software system for research in algebraic geometry.
\newblock Available at \url{http://www.math.uiuc.edu/Macaulay2/}.

\bibitem[Hil98]{hille}
L.~Hille.
\newblock Homogeneous vector bundles and {K}oszul algebras.
\newblock {\em Math. Nachr.}, 191:189--195, 1998.

\bibitem[KL18]{radical}
R.~Kinser and A.~C. L{\H{o}}rincz.
\newblock Representation varieties of algebras with nodes.
\newblock {\em arXiv}, 1810.10997, 2018.

\bibitem[KR15]{ryan}
R.~Kinser and J.~Rajchgot.
\newblock Type {A} quiver loci and {S}chubert varieties.
\newblock {\em J. Commut. Algebra}, 7(2):265--301, 2015.

\bibitem[Las78]{lasc}
A.~Lascoux.
\newblock Syzygies des vari\'et\'es d\'eterminantales.
\newblock {\em Adv. Math.}, 30(3):202--237, 1978.

\bibitem[LM98]{lak}
V.~Lakshmibai and P.~Magyar.
\newblock Degeneracy schemes, quiver schemes, and {S}chubert varieties.
\newblock {\em Int. Res. Res. Not.}, 1998(12):627--640, 1998.

\bibitem[L{\H{o}}r15]{bubsing}
A.~C. L{\H{o}}rincz.
\newblock Singularities of zero sets of semi-invariants for quivers.
\newblock {\em arXiv}, 1509.04170, 2015.
\newblock To appear in Journal of Commutative Algebra.

\bibitem[L{\H{o}}r17]{bfun}
A.~C. L{\H{o}}rincz.
\newblock The $b$-functions of semi-invariants of quivers.
\newblock {\em J. Algebra}, 482:346--–363, 2017.

\bibitem[L{\H{o}}r20]{bubdecomp}
A.~C. L{\H{o}}rincz.
\newblock Decompositions of {B}ernstein--{S}ato polynomials and slices.
\newblock {\em Transform. Groups}, 25(2):577--607, 2020.

\bibitem[LW19]{bubfree}
A.~C. L\H{o}rincz and J.~Weyman.
\newblock Free resolutions of orbit closures of {D}ynkin quivers.
\newblock {\em Trans. Amer. Math. Soc.}, 372(4):2715--2734, 2019.

\bibitem[OR06]{ottaviani}
G.~Ottaviani and E.~Rubei.
\newblock Quivers and the cohomology of homogeneous vector bundles.
\newblock {\em Duke Math. J.}, 132(3):459--508, 2006.

\bibitem[Rei03]{reineke}
M.~Reineke.
\newblock Quivers, desingularizations and canonical bases.
\newblock In {\em Studies in memory of {I}ssai {S}chur}, volume 210 of {\em
  Progress in Mathematics}, pages 325--344. Birkh\"auser, Boston, 2003.

\bibitem[Sut13]{kavita2}
K.~Sutar.
\newblock Resolutions of defining ideals of orbit closures for quivers of type
  {$\mathbb{A}_3$}.
\newblock {\em J. Commut. Algebra}, 5(3):441--475, 2013.

\bibitem[Sut15]{kavita}
K.~Sutar.
\newblock Orbit closures of source-sink {D}ynkin quivers.
\newblock {\em Int. Math. Res. Not.}, 2015(11):3423--3444, 2015.

\bibitem[Wey03]{jerzy}
J.~Weyman.
\newblock {\em Cohomology of vector bundles and syzygies}, volume 149 of {\em
  Cambridge Tracts in Mathematics}.
\newblock Cambridge University Press, Cambridge, 2003.

\bibitem[Zwa11]{expo}
G.~Zwara.
\newblock Singularities of orbit closures in module varieties.
\newblock In {\em Representations of Algebras and Related topics}, EMS Series
  of Congress Reports, pages 661--725. European Mathematical Society, 2011.

\end{thebibliography}

\end{document}